\documentclass[12pt,a4paper]{article}
\usepackage[top=3.5cm,bottom=2cm,right=3cm,left=3cm]{geometry}
\usepackage{graphicx}
\usepackage{color}
\usepackage{amsthm}
\usepackage{amssymb}
\usepackage{amsmath}
\usepackage{amsfonts}

\theoremstyle{plain}

\newtheorem{lem}{\emph{Lemma}}
\newtheorem{prop}{\emph{Proposition}}
\newtheorem*{defi}{\emph {Definition}}
\newtheorem*{rmk}{\emph{Remark}}
\everymath{\displaystyle}

\newcommand{\al}{\alpha}

\newcommand{\lam}{\lambda}

\newcommand{\ep}{\epsilon}
\newcommand{\del}{\delta}

\newcommand{\ra}{\rightarrow}

\newcommand{\bN}{\mathbb{N}}
\newcommand{\bZ}{\mathbb{Z}}

\newcommand{\bR}{\mathbb{R}}
\newcommand{\bC}{\mathbb{C}}

\begin{document}

 \title{\textrm On the Roots of Characteristic Equations of Delay Differential Systems} 
 \author{Jia-Yuan Dai} 
 \date{ } 
 \maketitle


\begin{abstract}
We prove that characteristic equations of certain types of delay differential systems, under some mild conditions on their coefficients, can possess infinitely many complex roots.  
\end{abstract}

\noindent \textbf{A. Prelimilary}
\\
\\
Our motivation comes from the linear (single, complex, constant) time-delay complex differential system:
\begin{equation} \label{dde}
\dot{x}(t) = Ax(t) + Bx(t-\tau), \quad x(t) \in \bC^n
\end{equation} 
where $A$ and $B$ are $n$-by-$n$ matrices over $\bC$ and $\tau \in \bC \setminus \{0\}$ is a complex time-delay. The stablity of the zero solution is determined to the real parts of roots of the characteristic equation:
\begin{equation} \label{1}
f(\lam) : = \mathrm{det}\big( \lam \mathrm{id} - A - e^{- \tau \lam} B \big) = 0,
\end{equation}
after the exponential ansatz $x(t) = e^{\lam t} x_0$ is applied. We are interested in the question whether there exist infinitely many complex roots of $f$.
\\ 
\indent Our main observations are the following:
\begin{itemize}
\item[(i)] $f$ is an entire function;
\item[(ii)] for any $\ep > 0$, the growth rate of $f$ is bounded by $e^{|\lam|^{1+\ep}}$ for all $\lam \in \bC$ with $|\lam|$ sufficiently large. 
\end{itemize}
We note that (ii) follows directly by using triangle inequality.
\begin{defi}
Let $f$ be an entire function, the \textbf{order} of $f$, denoted by $\mathrm{ord}(f)$, is the infimum of $\al > 0$ such that there exists $R > 0$ such that $|f(\lam)| \le e^{|\lam|^\al}$ holds for all $ \lam \in \bC$ with $|\lam| \ge R$.
\end{defi}
Hence the observation (ii) indicates that $\mathrm{ord}(f) \le 1$. Now finiteness of $\mathrm{ord}(f)$ reminds us a dichotomy.
\begin{lem}[Theorem 16.13 in \cite{BaNe10}]  \label{lem1}
Let $f$ be an entire function and of finite order, then 
\begin{itemize}
\item[(i)] either $f(\lam) = 0$ possesses infinitely many roots in $\bC$, 
\item[(ii)] or there exist complex polynomial $g(\lam)$ and $h(\lam)$ such that $h(0) = 0$ and
$$
f(\lam) = g(\lam) e^{h(\lam)}
$$
holds for all $\lam \in \bC$.
\end{itemize}
Furthermore, in the case (ii), we have $\mathrm{deg}(h) = \mathrm{ord}(f)$.
\end{lem}
Thus, our strategy is to give a indirect proof: according to Lemma \ref{lem1}, if $f(\lam) = 0$ possesses at most finitely many roots in $\bC$, then $\mathrm{ord}(f) \le 1$ implies 
$$
f(\lam) = g(\lam) e^{c \lam}
$$ 
holds for all $\lam \in \bC$ where $g(\lam)$ is a complex polynomial  and $c \in \bC$. Then the main task is to seek conditions on the coefficients $A$ and $B$ to reach a contradiction.
\\
\\
\noindent \textbf{B. Single Complex Constant Delay}
\\
\\
In the following Proposition we apply our strategy carefully.
\begin{prop} \label{prop1}
Suppose $\mathrm{tr}(B) \neq 0$, then for each $\tau \in \bC \setminus \{0\}$, the equation
\begin{equation} \label{1eq1}
f(\lam) := \mathrm{det}\big( \lam \mathrm{id} - A - e^{- \tau \lam} B \big) = 0
\end{equation}
possesses infinitely many roots in $\bC$.
\end{prop}

\begin{proof}[\textbf{Proof.}] 
Setting $ \lam \mapsto \tau \lam$, without loss of generality we consider $\tau = 1$. The equation (\ref{1eq1}) can be expressed as
\begin{equation} \label{2}
f(\lam) := \lam^n + a_1(\lam) e^{-\lam} + ... + a_n(\lam) e^{-n \lam},
\end{equation}
where 
$$
a_1(\lam) = - (\mathrm{tr}(B)) \lam^{n-1} + \mbox{lower order terms}
$$ 
is a nonzero polynomial since we assume $\mathrm{tr}(B) \neq 0$. Obviously all other $a_j(\lam)$ for $j \in \{2,...,n\}$ are (maybe identically zero) complex polynomials. Since $a_1(\lam)$ is nonzero, there exist $k \in \bN$ with $1 \le k \le n$ such that $a_k(\lam)$ is the \textit{last} (with respect to the order as real numbers in the exponential exponents) nonzero polynomial, i.e.
\begin{equation} \label{1}
f(\lam) = \lam^n + a_1(\lam) e^{-\lam} + ... + a_k(\lam) e^{-k \lam}.
\end{equation}
Obviously $f$ is an entire function. We easily see that $\mathrm{ord}(f) \le 1$, because for each $\ep >0$, using triangle inequality, the estimates
\begin{equation} \label{order1}
|f(\lam)| \le (k+1) \max_{j = 1,...,k}\{ 1, |a_j(\lam)| \} e^{k |\lam|} \le e^{|\lam|^{1+\ep}}
\end{equation}
hold as $|\lam|$ is sufficiently large. 
\paragraph{Contradiction Part:} Suppose the contrary that $f(\lam) = 0$ possesses at most finitely many roots in $\bC$. Since $f$ is entire and $\mathrm{ord}(f) \le 1$, by Lemma \ref{lem1},
$$
f(\lam) = g(\lam) e^{c \lam}
$$
holds for all $\lam \in \bC$ where $g(\lam)$ is a complex polynomial and $c \in \bC$. We claim that 
\begin{equation} \label{contradiction}
\mathrm{Re}(c) = -k, \quad \mathrm{Im}(c) = 0.
\end{equation}
\indent Let $z_l \lam^l$ ($0 \le l \le n$) be the leading term of $a_k(\lam)$. Multiplying (\ref{1}) by $e^{k\lam}/\lam^l$ yields
\begin{equation} \label{3}
\frac{ g(\lam) }{\lam^l} e^{i \mathrm{Im}(c) \lam} e^{(\mathrm{Re}(c) + k) \lam} = \frac{ \lam^n e^{k \lam} + a_1(\lam)e^{(k-1)\lam} + ... a_{k-1}(\lam) e^{\lam}}{\lam^{l}} + z_l +  \frac{\tilde{a}_k(\lam)}{\lam^l} 
\end{equation}
where $\mathrm{deg}(\tilde{a}_k) < l$. Taking $\lam \in \bR$ and $\lam \ra -\infty$, since $g$, all $a_j$, and $\tilde{a}_k$ are polynomials, the right-hand side of (\ref{3}) converges to $z_l$, while the left-hand side of (\ref{3}) diverges to infinity (resp. to zero) if $\mathrm{Re}(c) + k < 0$ (resp. $\mathrm{Re}(c) + k > 0$). Thus $
\mathrm{Re}(c) + k = 0$. We now have
\begin{equation} \label{p13}
\frac{ g(\lam) }{\lam^l} e^{i \mathrm{Im}(c) \mathrm{Re}(c)} e^{- \mathrm{Im}(c) \mathrm{Im}(\lam)} = \frac{ \lam^n e^{k \lam} + a_1(\lam)e^{(k-1)\lam} + ... a_{k-1}(\lam) e^{\lam}}{\lam^{l}} + z_l +  \frac{\tilde{a}_k(\lam)}{\lam^l}. 
\end{equation}
Again we play the same trick by taking $\mathrm{Re}(\lam) \ra -\infty$ and $\mathrm{Im}(\lam) \ra \infty$ (or $-\infty$, it does not matter),  we see $\mathrm{Im}(c) = 0$. As a result, (\ref{1}) becomes 
$$
g(\lam) e^{-k \lam} = \lam^n + a_1(\lam) e^{-\lam} + ... + a_k(\lam) e^{-k \lam}.
$$
At last taking $\lam \in \bR$ and $\lam \ra \infty$ we have
$$
0 = \lim_{\lam \in \bR, \, \lam \ra \infty} \lam^n,
$$
which is a contradiction. The proof is complete.
\end{proof}
%
\noindent \textbf{C. Multiple Real Constant Delays}
\\
\\
We consider the linear (multiple, real, constant) time-delay complex differential systems:
\begin{equation} \label{dde2}
\dot{x}(t) = Ax(t) + \sum_{j =1}^k B_j x(t-\tau_j), \quad x(t) \in \bC^n
\end{equation} 
for integer $j \ge 2$ and $-\infty < \tau_1 < \tau_2 < ... < \tau_k < \infty$. The characteristic equation is given by 
\begin{equation} \label{multiple}
\mathrm{det}\big(\lam \mathrm{id} - A - \sum_{j=1}^k B_j e^{-\tau_j \lam}\big) = 0,
\end{equation}
which is a special case of the general \textbf{\textit{quasi-polynomials}}
\begin{equation} \label{form}
f(\lam) := \sum_{(\al_0, \al_1, ..., \al_k) \in \bN^{k+1} \cup \{0\}} a_{\al_0, \al_1,...,\al_k} \lam^{\al_0} e^{- ( \sum_{j=1}^k \al_j \tau_j) \lam}
\end{equation}
where only finitely many $a_{\al_0, \al_1,...,\al_k} \in \bC$ are nonzero. Denote $\boldsymbol{\tau} := (\tau_1, ..., \tau_k)$ and $\boldsymbol{\al} := (\al_1,..., \al_k)$. We call $f$ is \textbf{\textit{admissible}} if there exist $\boldsymbol{\al^1}$ and $\boldsymbol{\al^2}$ such that 
$$
 \boldsymbol{\al^1} \cdot \boldsymbol{\tau} \neq \boldsymbol{\al^2} \cdot \boldsymbol{\tau} 
$$
and there exist $\al_0^1, \al_0^2 \in \bN \cup \{0\}$ such that 
$$
a_{\al_0^1, \boldsymbol{\al^1}} \neq 0, \quad a_{\al_0^2, \boldsymbol{\al^2}} \neq 0.
$$
In other words, $f(\lam)$ possesses two different exponential exponents.
\begin{prop} \label{prop2}
Let $f$ be defined in (\ref{form}), then $f(\lam) = 0$ possesses infinitely many roots in $\bC$ if and only if $f$ is admissible. 
\end{prop}
\begin{proof}[\textbf{Proof.}]
Assume $f$ is not admissible, then $f(\lam) = 0$ is equivelent to a polynomial equation, which possesses at most finitely many roots in $\bC$.
\\
\indent Conversely, assume $f$ is admissible. Obviously $f$ is an entire function and $\mathrm{ord}(f) \le 1$. Since all $\tau_j$ are real, the terms of $f(\lam)$ can be sorted by the order as real numbers in the exponential exponents. Hence if $f$ is admissible, then 
\begin{equation} \label{2eq}
f(\lam) = a_h(\lam) e^{- (\boldsymbol{\al^h} \cdot \boldsymbol{\tau}) \lam} + ... + a_l(\lam) e^{- (\boldsymbol{\al^l} \cdot \boldsymbol{\tau}) \lam}
\end{equation}
holds where $a_h(\lam)$ and $a_l(\lam)$ are nonzero complex polynomials and $- (\boldsymbol{\al^h} \cdot \boldsymbol{\tau}) > - (\boldsymbol{\al^l} \cdot \boldsymbol{\tau})$ are two different real numbers.
Therefore, $f(\lam) = 0$ is equivalent to the equation
$$
\tilde{f}(\lam) = a_h(\lam) + ... + a_l(\lam) e^{-(\boldsymbol{\al^l } \cdot \boldsymbol{\tau} -\boldsymbol{\al^h } \cdot \boldsymbol{\tau}) \lam } = 0.
$$
\paragraph{Contradiction Part:} Suppose the contrary that $\tilde{f}(\lam) = 0$ possesses at most finitely many roots in $\bC$, then 
\begin{equation} \label{prop2eq}
\tilde{f}(\lam) = g(\lam) e^{c \lam}
\end{equation}
holds for all $\lam \in \bC$. Now we notice that the claim in the Contradiction Part of the previous Proposition:
$$
\mathrm{Re}(c) =-(\boldsymbol{\al^l } \cdot \boldsymbol{\tau} -\boldsymbol{\al^h } \cdot \boldsymbol{\tau}), \quad \mathrm{Im}(c)= 0,
$$
holds if we assume all $\tau_j$ are real. Therefore, taking $\lam \in \bR$ and $\lam \ra \infty$ in (\ref{prop2eq}), we have
$$
0 = \lim_{\lam \in \bR, \, \lam \ra \infty} a_h(\lam),
$$
a contradiction. The proof is complete.
\end{proof}
\begin{rmk}
The assumption $\mathrm{tr}(B) \neq 0$ is just a sufficient condition of Proposition \ref{prop1}, but it is the unique sufficient condition that is irrelevant to $A$. 
\end{rmk}
\begin{rmk}
It is interesting to seek sufficient conditions for $f$ in (\ref{multiple}) being admissible. For instance Pontryagin's condition that $f$ is without the principal term, see \cite{Po55}. Another sufficient condition is that $\tau_j$ are linearly independent over $\bZ$, i.e. 
$$
\boldsymbol{\beta} \cdot \boldsymbol{\tau} =0, \quad \boldsymbol{\beta}\in \bZ^k \quad \mbox{implies} \quad \boldsymbol{\beta} = \textbf{0}.
$$
and one of $B_j$ is of trace zero.
\end{rmk}
%
%
\noindent \textbf{D. Single Real Distributed Delay}
\\
\\
We consider a linear (single, real, distributed) time-delay complex differential equation:
\begin{equation} \label{d.eq}
\dot{x}(t) = ax(t) + \int_0^{\tau} M(\theta) x(t - \theta) d \theta, \quad x(t) \in \bC.
\end{equation} 
where $a \in \bC$, $\tau > 0$, and $M \in C^0([0, \tau], \bC)$. The characteristic equation of (\ref{d.eq}) is given by
$$
f(\lam) := \lam - a - \int_0^{\tau} M(\theta) e^{-\lam \theta} d \theta = 0
$$
\begin{prop} \label{prop3}
$f(\lam) = 0$ possesses infinitely many roots in $\bC$ if and only if $M$ is not identically zero.
\end{prop}
\begin{proof}[\textbf{Proof.}]
Assume $M$ is identically zero, then $f(\lam) = 0$ possesses the unique root $\lam = a$.
\\
\indent Conversely, assume $M$ is not identically zero. Suppose the contrary that $f$ possesses at most finitely many roots in $\bC$. Obviously $f$ is an entire function and $\mathrm{ord}(f) \le 1$, then by Lemma \ref{lem1}, 
$$
f(\lam) = g(\lam) e^{c \lam}
$$
holds for all $\lam \in \bC$. Define $\del := \tau \|M\|_{C^0} > 0$, then using triangle inequality,
\begin{equation} \label{growth}
|\lam - a| - \del e^{\tau \mathrm{Re}(\lam)} \le |f(\lam)| = |g(\lam)| e^{\mathrm{Re}(c) \mathrm{Re}(\lam) - \mathrm{Im}(c)\mathrm{Im}(\lam)} \le |\lam - a| + \del e^{\tau \mathrm{Re}(\lam)}.
\end{equation}
Taking $\lam \in \bR$ and $\lam \ra -\infty$, we see $|f(\lam)|$ cannot grow exponentially, hence $\mathrm{Re}(c) = 0$. Similarly, taking $\mathrm{Re}(\lam) \ra -\infty$ and $\mathrm{Im}(\lam) \ra \infty$ (or $-\infty$, it does not matter), we have $\mathrm{Im}(c) = 0$. Now that $c = 0$, the growth constraint (\ref{growth}) of $|g(\lam)|$ also implies that $g(\lam)$ is linear. Therefore there exist $p, q \in \bC$ such that 
\begin{equation} \label{zero}
\int_0^{\tau} M(\theta) e^{-\lam \theta} d \theta = p \lam + q
\end{equation}
holds for all $\lam \in \bC$. To reach a contradiction, we differentiate (\ref{zero}) twice to obtain
$$
\int_0^{\tau} \theta^2 M(\theta) e^{-\lam \theta} d \theta = 0.
$$
Since $M$ is continuous, by using Fourier series, we have $\theta^2M(\theta) = 0$ for all $\theta \in [0, \tau]$. Thus $M$ is identically zero, a contradiction. The proof is complete.
\end{proof}



\end{document}